\documentclass[11pt]{article}
\usepackage{geometry}
\usepackage{amsmath, amssymb, amsthm}
\usepackage[caption = false]{subfig}
\usepackage{graphicx}
\usepackage{color}
\usepackage[utf8]{inputenc}
\usepackage{latexsym}
\newtheorem{theorem}{Theorem}[section]
\newtheorem{lemma}{Lemma}[section]
\usepackage{hyperref}
\usepackage{graphicx,epsfig}
\usepackage{color}
\usepackage{epstopdf}
\usepackage{caption}
\usepackage{float}
\usepackage{enumitem}   
\usepackage{authblk}

\newcommand{\IR}{\mathbb{R}}

\newcommand{\IO}{\mathcal{O}}

\begin{document}

\title{Normal form for singular Bautin bifurcation in a slow-fast system with Holling type III functional response}

\author[$^1$]{Tapan Saha}
\author[$^2$]{Pranali Roy Chowdhury}
\author[$^3$]{Pallav Jyoti Pal}
\author[$^2$]{Malay Banerjee}

\affil[$^1$]{Department of Mathematics, Presidency University, Kolkata, India}
\affil[$^2$]{Department of Mathematics and Statistics, IIT Kanpur, Kanpur, India }
\affil[$^3$]{Department of Mathematics, Krishna Chandra College, Hetampur, India}

\maketitle

\begin{abstract}
Over the last few decades, complex oscillations of slow-fast systems have been a key area of research. In the theory of slow-fast systems, the location of singular Hopf bifurcation and maximal canard is determined by computing the first Lyapunov coefficient. In particular, the analysis of canards is based on the genericity condition that the first Lyapunov coefficient must be non-zero. This manuscript aims to further extend the results to the case where the first Lyapunov coefficient vanishes. For that, the analytic expression of the second Lyapunov coefficient and the investigation of the normal form for codimension-2 singular Bautin bifurcation in a predator-prey system is done by explicitly identifying the locally invertible parameter-dependent transformations. 
A planar slow-fast predator-prey model with Holling type III functional response is considered here, where the prey population growth is affected by the weak Allee effect and the prey reproduces much faster than the predator. Using geometric singular perturbation theory, normal form theory of slow-fast systems, and blow-up technique, we provide a detailed mathematical investigation of the system to show a variety of rich and complex nonlinear dynamics including but not limited to the existence of canards, relaxation oscillations, canard phenomena, singular Hopf bifurcation, and singular Bautin bifurcation. 
Additionally, numerical simulations are conducted to support the theoretical findings. 
\end{abstract}



\section{Introduction}

Slow-fast systems or singularly perturbed systems of ordinary differential equations evolving on two or more timescales are represented in the form 
\begin{eqnarray}\label{general_slow_fast}
	\dot{x}\, =\, f\left(x, y, \epsilon\right),\,\,
\dot{y} \,=\, \epsilon g\left(x, y, \epsilon\right),
 \end{eqnarray}
where $x \in \mathbb{R}^m$ and $y \in \mathbb{R}^n$ ($m,n\geq 1$) are the fast and the slow variables respectively, $f$ and $g$ are sufficiently smooth functions,  $0 < \epsilon \ll 1$ is the singular perturbation parameter, and the over dot $(~ \dot{}~ )$ stands for derivative w.r.t. $t \in \mathbb{R}$. They are known to exhibit a wide range of complicated oscillations of many physical and applied systems including, but not limited to, canard cycles, relaxation oscillations, and mixed mode oscillations \cite{kuehn2015multiple}. The investigation of these systems has received considerable attention in recent years \cite{fitzhugh1955mathematical,han2012slow,hsu2021relaxation,kristiansen2019geometric,krupa2008mixed,krupa2001relaxation,li2018bifurcation,qinsheng2020relaxation,yaru2020canard}. 

In this article, we concentrate on slow-fast predator-prey systems of the type \eqref{general_slow_fast} with $m=n=1$. Some examples of recent research in this field are as follows:
       The author in \cite{hek2010geometric} discussed and analysed the geometric singular perturbation theory, three major Fenichel theorems, and their relevance to biological modelling applications. 
        In \cite{kooi2018modelling}, the traditional Rosenzweig-MacArthur (RM) model and the Mass Balance (MB) chemostat model are investigated in the slow-fast framework. According to their analysis, the RM model captures the dynamically rich canard explosion phenomenon, but the MB model does not. Consequently, in the slow-fast setting, the predictions of the MB Model are quite different from that of the RM model. 
		The canard phenomenon and number of canard cycles for a predator-prey system with Holling type III and IV functional response was investigated by Li and Zhu \cite{li2013canard}.
		 In \cite{atabaigi2021canard} the author studies the dynamics of a multiple timescale predator-prey model where the predator is a generalist who feeds on both the focal prey and the functional response is Holling type III. With the help of the normal form theory of slow-fast systems, the geometric singular perturbation theory, and blow-up technique, the author \cite{krupa2001extending} has investigated the existence of homoclinic orbits, heteroclinic orbits, canard limit cycles, and relaxation oscillations bifurcating from the singular homoclinic cycles.

In order to systematically analyze a slow-fast system, it is important to recognize the different subprocesses occurring at the various time scales, comprehend them, and then attempt to characterize the complete dynamics of the full system based on the dynamical behaviour of the subsystems. The geometric analysis pioneered by Fenichel \cite{fenichel1979geometric} is based on this approach. Normal hyperbolicity is a crucial property that critical manifolds may have.
The mathematical theory, known as geometric singular perturbation theory (GSPT) is used to analyze the systems of the form \eqref{general_slow_fast} \cite{dumortier1996canard,fenichel1979geometric,krupa2001relaxation,kuehn2015multiple}. To analyse slow-fast systems with hyperbolic points, one may make use of the geometrical tools and methods outlined by Fenichel theory  \cite{fenichel1979geometric}. However, other geometrical approaches, such as the blow-up method, introduced by Dumortier and Roussarie \cite{dumortier1996canard} and developed by Krupa and Szmolyan \cite{krupa2001extending,krupa2001relaxation}, and slow-fast normal form theory \cite{arnold1994dynamical}, are commonly used to explore the dynamics of a slow-fast system with non-hyperbolic singularities. The blow-up technique includes de-singularization of non-hyperbolic singularities. The key idea of such blow-up transform is to reduce a slow-fast system  into the form 
      \begin{align}
       x' = f(x, y, \mu),~~ y' = \epsilon g(x, y, \mu), ~~\epsilon' = 0, 
    \end{align}
   
such that a non-hyperbolic equilibrium, say, $(x, y,\epsilon) = (0, 0, 0)$ (where $D_xf(0,0,0)$ has all its eigenvalues equal to zero), is blown-up, for example, to a sphere, to desingularize the equilibrium $(0, 0, 0)$. Generally, different (overlapping) directional charts are used to investigate the dynamics for the blown-up system and finally, a
qualitative description of the original system is obtained by blow-down. Using the above-mentioned technique, one can identify the existence of canards, relaxation oscillations, canard phenomena, singular Hopf bifurcation, singular Bautin bifurcation.

In this article, we consider a slow-fast predator-prey system with a weak Allee effect in the prey ($x$) population growth and sigmoid functional response of the form $f(x)=\frac{qx^2}{x^2+c}$, called Holling type-III, where $c>0$ is the half-saturation constant for the predator $y$ that gauges how abruptly $f(x)$ changes \cite{getz1996hypothesis, huang2014bifurcations}. Allee effect comprises both weak and strong Allee effects \cite{bai2021dynamics,van2007heteroclinic}. We note that the weak Allee effect growth function, in contrast to the strong Allee effect, is always positive and has no threshold value  \cite{courchamp2008allee,pal2012delayed,wang2011predator}. In this article, the weak Allee effect is incorporated into the prey population growth, allowing the prey to grow even at low population densities. Based upon the above considerations, the model can be described as follows: 
\begin{subequations}\label{model_dim}
\begin{eqnarray}
\frac{dx}{dT}&=&rx(1-\frac{x}{K})(x+m)-\frac{qx^2y}{x^2+c},\\
\frac{dy}{dT}&=&\frac{px^2y}{x^2+c}-dy,
\end{eqnarray}
\end{subequations}
with $(c,d, K, m,p,q, r) \in \mathbb{R}^7_+$, where $x = x(t)$ and $y = y(t)$ denote the prey and predator population densities at time $t > 0$, 
$c>0$ is the half-saturation constant for the predator,
$d$ is the predator's natural death rate,
$K$ is the environmental carrying capacity, 
$r$ is the intrinsic per capita growth
rate of prey, $q$ denotes the maximum per capita consumption rate, and $p$ denotes the conversion efficiency of consumed prey into new predators \cite{getz1996hypothesis,hsu2001global,hsu2008global}.
Considering that the prey reproduces substantially much faster than the predator and using the following scaling of the variables given by,
\begin{equation*}
     x=Ku, ~~ y=\frac{rK^2}{q}v, ~~T=\frac{t}{rK},
\end{equation*}
the model system \eqref{model_dim} can then
be rewritten in the following slow-fast dimensionless form: 
\begin{subequations}\label{temp_model}
\begin{eqnarray}
\frac{du}{dt}&=&\left((1-u)(u+\theta)-\frac{uv}{u^2+\eta}\right)u,\\
\frac{dv}{dt}&=&\epsilon\left(\frac{u^2}{u^2+\eta}-\delta\right)v,
\end{eqnarray}
\end{subequations}
where $\mu=(\delta,\theta,\eta)\in \mathbb{R}_+\times (0,1)^2$ and $0<\epsilon\ll 1$ and subjected to the initial conditions $u(0)>0$ and $v(0)>0$ such that $\delta=\frac{d}{p}$, $\theta=\frac{m}{K}$, $\eta=\frac{c}{K^2}$ and $\epsilon=\frac{d}{p}$ are the dimensionless parameters.

The main aim of this article is to analyze the phase-space of predator–prey systems \eqref{temp_model} in the context of a slow-fast framework and offer a detailed description of the vast range of rich and complex local and global dynamical behaviour, including, but not just restricted to, singular Hopf bifurcations, singular Bautin bifurcations, canard cycles, canard phenomenon and occurrence of relaxation oscillations via canard cycles known as ``boom and bust cycle" \cite{rinaldi1992slow,sadhu2022analysis}. It suggests that both populations may coexist with a predictable pattern of population explosions and contractions. The main motivation of this work is the derivation of the analytic form of the second Lyapunov coefficient and a thorough study of the singular Bautin bifurcation for a system under slow-fast framework \cite{kuznetsov1998elements, lu2021global}. The normal form for singular Bautin bifurcation has been derived here by explicitly finding the locally invertible parameter-dependent transformations. So far as our knowledge goes, the explicit derivation of locally invertible parameter-dependent transformations is not available in the literature. 

The remaining portion of the paper is organized as follows: Section \ref{sec2} reports the invariance, boundedness of the solution,  existence, and local stability analysis of equilibrium points of the system \eqref{temp_model}. The slow-fast system is analyzed by separating system \eqref{temp_model} into the fast and slow limiting subsystems in Section \ref{sec3}. The singular bifurcations (singular Hopf and singular Bautin) results are described in \ref{sec4}. We investigate the canard explosion and relaxation oscillation in Section \ref{sec5}. In Section \ref{sec6}, numerical simulations are performed to verify the major theoretical predictions. In the last Section \ref{sec7}, we discuss our findings and provide our conclusions. 

\section{Positivity and boundedness of the solution}\label{sec2}

We state the following lemmas to ensure the positivity and boundedness of the solution for the system \eqref{temp_model}. \textcolor{black}{The proofs of the lemmas are omitted as it directly follows from the approach discussed in \cite{BirkhoffRota}. } 
\begin{lemma}\label{lemma_1}
The first quadrant $\mathbb{R}_+=\{(u,v)\in\IR^2|u\geq 0, v\geq 0\}$ is invariant under the flow generated by the vector field 
$\mathcal{F}= \hat{f}\frac{\partial}{\partial u}+\epsilon \hat{g}\frac{\partial}{\partial v}$ where $\hat{f} = u(1-u)(u+\theta)-\frac{u^2v}{u^2+\eta}$ and $\hat{g} = \frac{u^2v}{u^2+\eta}-\delta v.$
\end{lemma}



\begin{lemma}\label{lemma_2}
All the solutions of the model system \eqref{temp_model} initiated from an interior point of $\IR^2_+$ are bounded.
\end{lemma}


\subsection{Equilibrium points and their local stability}
The equilibria of the system \eqref{temp_model} are given by the interaction of the prey and predator nullcline as follows,
$$(1-u)(u+\theta)(u^2+\eta)\,=\,uv,\,\, u=0,$$ 
and
$$u^2/(u^2+\eta)=\delta,\,\,v=0.$$
The system has $E_0=(0,0)$ and $E_1=(1,0)$ as the boundary equilibria. The interior equilibrium points are obtained when the non-trivial prey and predator nullclines intersect in the interior of $\IR^2_+.$ This gives us at most one feasible interior equilibrium point $E_*=(u_*,v_*)$ where
\begin{align*}
 u_* = \sqrt{\delta\eta/(1-\delta)},\,\,\,\,\,\,v_*=(1-u_*)(u_*+\theta)(u_*^2+\eta)/u_*.  
\end{align*}
The equilibrium $E_*$ is feasible if $\delta<\frac{1}{1+\eta}$. Linearizing \eqref{temp_model} in the neighbourhood of $E_0$ we obtain that the eigenvalues of the Jacobian matrix are $0$ and $-\epsilon\delta.$
 Because of the zero eigenvalue, we cannot directly conclude about the dynamics of the system \eqref{temp_model} near the origin. The eigenvector corresponding to $0$ is the $u$ axis and the eigenvector corresponding to $-\epsilon\delta$ is the $v$ axis. Therefore, any trajectory starting very close to the $v$ axis gets attracted toward the origin and repels away from the vicinity of the origin. The origin is thus an attracting saddle node. We further linearize \eqref{temp_model} in the neighbourhood of $E_1$ and the eigenvalues of the Jacobian matrix are $-1-\theta<0$ and $\epsilon\Big(\frac{1}{1+\eta}-\delta\Big).$ 
For $\frac{1}{1+\eta}<\delta$ the equilibrium $E_1$ is a stable node, for $\frac{1}{1+\eta}>\delta$ the equilibrium is a saddle point, and for $\delta=\frac{1}{1+\eta}$ the equilibrium $E_1$ is an attracting saddle node. The eigenvector corresponding to the eigenvalue $-1-\theta$ is the $u$ axis. Thus, combining we get, the trajectory starting near the vertical axis gets attracted towards the origin, and then gets attracted towards $E_1.$
The stability of the equilibrium point $E_*$ depends on the geometry of the non-trivial prey nullcline and the position of $E_*$ on it. If there does not exist any feasible coexistence equilibrium point $E_*$ then the boundary equilibrium $E_1$ is the global attractor, otherwise, it is a saddle point. The system \eqref{temp_model} thus undergoes a transcritical bifurcation at $\delta=\frac{1}{1+\eta}$. The stability of the interior equilibrium point will be discussed in the upcoming sections.



\section{Analysis of slow-fast system}\label{sec3}
We consider a topologically equivalent system of the slow-fast system \eqref{temp_model} by re-scaling the time $dt \rightarrow (u^2+\eta)dt,$ such that $(u^2+\eta)>0$ as follows
\begin{equation}\label{poly_sys}
    \begin{aligned}
     \frac{du}{dt} &=u(1-u)(u+\theta)(u^2+\eta)-u^2v=f(u,v),\\
     \frac{dv}{dt} &= \epsilon(u^2v-\delta v(u^2+\eta))=\epsilon g(u,v).
    \end{aligned}
\end{equation}
The parameters are positive with $0<\epsilon\ll 1$.
With the transformation $\tau:=\epsilon t$, the system \eqref{poly_sys} is equivalent to 
\begin{equation}\label{eqn4}
    \begin{aligned}
        \epsilon \frac{du}{d\tau} &= \left((1-u)(u+\theta)(u^2+\eta)-uv\right)u,\\
        \frac{dv}{d\tau} &= \left(u^2 -\delta (u^2+\eta)\right)v.
    \end{aligned}
\end{equation}
The systems \eqref{poly_sys} and \eqref{eqn4} are referred to as fast and slow systems with $`t'$ as the fast timescale and $`\tau'$ as the slow timescale. The variables $u$ and $v$ are known as fast and slow variables. In the limiting sense of the fast system, that is, as $\epsilon=0,$ we obtain the layer systems as
\begin{equation}
    \begin{aligned}
          \frac{du}{dt} = \left((1-u)(u+\theta)(u^2+\eta)-uv\right)u,\,\,
        \frac{dv}{dt} = 0,
    \end{aligned}
\end{equation}
and the slow system is reduced to
\begin{equation}
    \begin{aligned}
         0 = \left((1-u)(u+\theta)(u^2+\eta)-uv\right)u,\,\,
        \frac{dv}{d\tau} = \left(u^2 -\delta (u^2+\eta)\right)v.
    \end{aligned}
\end{equation}
The fast flow is along the horizontal line $v=$ constant, whereas the slow flow is confined to the set 
\begin{align*}
    M_{10}& = \{(u,v)\in \mathbb{R}^2_+: u=0 \},\\
    M_{20} &= \{(u,v)\in \mathbb{R}^2_+: v=\phi(u)= \frac{1}{u}(1-u)(u+\theta)(u^2+\eta) \},
\end{align*}
known as critical manifold. The critical manifold $M_{20}$ is a one-dimensional curve that is normally hyperbolic whenever $\phi'(u)>0$ or $<0.$ The points at which $\phi'(u)=0$ are the fold points of the critical manifold. From $\phi'(u)=0,$ we obtain,
\begin{equation}\label{dphi}
    3u^4-2(1-\theta)u^3-(\theta-\eta)u^2+\eta\theta=0.
\end{equation}
Therefore, the fold points of the critical manifold are the roots of the above quartic equation (if any). We claim that \textcolor{black}{under any of the following parametric restrictions} \textbf{P1:} $\eta<\theta \leq 1,$ \textbf{P2:} $\eta\leq \theta<1,$ \textbf{P3:} $\theta>\max\{1, \eta\},$ and \textbf{P4:} $\theta<\min\{1,\eta\}$ the critical manifold $M_{20}$ can have at most two-feasible fold points where the $u$-components are given by the positive solution of the quartic equation \eqref{dphi} and  $v=\frac{1}{u}(1-u)(u+\theta)(u^2+\eta).$  Let $F(u)=3u^4-2(1-\theta)u^3-(\theta-\eta)u^2+\eta\theta,$ then $F(0) = \eta\theta>0$ and $F(1) = (\theta+1)(\eta+1)>0.$  We argue that if the function $F$ has a minimum for some $u_{\mathrm{min}}\in (0,1)$ and $F(u_{\mathrm{min}})<0,$ then by continuity of the function $F$ we can say that there exist exactly two roots of the equation $F=0$ which corresponds to two fold points. For this we differentiate $F$ with respect to $u$ and from $\frac{dF}{du}=0$ we obtain,
$$u = 0,\, \frac{1-\theta}{4} \pm \frac{\sqrt{9\theta^2+6\theta+9-24\eta}}{12}.$$ 
\textcolor{black}{We first consider the case \textbf{(P1)} and assume $ \Gamma:= 9\theta^2+6\theta+9-24\eta>0$. This implies}
$$ 0< \frac{1-\theta}{4} + \frac{\sqrt{\Gamma}}{12} <1\,\,\text{and}\,\, \frac{1-\theta}{4} - \frac{\sqrt{\Gamma}}{12}<0.$$
We then claim that the function $F$ attains its minimum at $u_{\mathrm{min}}=\frac{1-\theta}{4} + \frac{\sqrt{\Gamma}}{12}.$ This gives $$\frac{d^2F}{du^2}\Big{|}_{u_{\mathrm{min}}} = \frac{\Gamma}{6} + \frac{(1-\theta)}{2}\sqrt{\Gamma} >0,\,\, \text{and}\,\, F(u_{\mathrm{min}}) = \Lambda_1 - \Lambda_2 \sqrt{\Gamma},$$ 
$$\text{where}\,\, \Lambda_1 = \frac{\eta}{8}(\theta^2+\frac{22}{3}\theta+1)-\frac{1}{96}(3(1+\theta^4)+2(\theta^2+4\eta^2)),\,\, \Lambda_2 = \frac{1-\theta}{288}(3\theta^2+2\theta+3-8\eta).$$
Therefore there exist exactly two feasible roots of the polynomial $F(u)$ in the parametric region $\mathcal{R}_1,$ where $$\mathcal{R}_1=\{(\eta,\theta): 0<\eta<\theta\leq 1,\,\, \frac{\sqrt{\Gamma}}{3}-3<\theta<1+\frac{\sqrt{\Gamma}}{3},\,\, \Gamma> \left(\frac{\Lambda_1}{\Lambda_2}\right)^2\},$$
and the two extrema of the curve $\phi(u)$ correspond to two fold points of the critical manifold $M_{20}$. The case \textbf{(P2)} is similar to case \textbf{(P1)}. Under the restriction \textbf{(P3)} we consider sub-cases, either $\eta<1<\theta$ or $1<\eta<\theta.$ We consider the first subcase. Since for $\eta<1<\theta,\,\, \frac{1-\theta}{4} - \frac{\sqrt{\Gamma}}{12}<0.$ Thus two fold points exist for $0< \frac{1-\theta}{4} + \frac{\sqrt{\Gamma}}{12} <1$ which implies $\frac{\sqrt{\Gamma}}{3}-3<\theta<1+\frac{\sqrt{\Gamma}}{3},$ and the minimum value must be negative, that is, $F(u_{\mathrm{min}}) = \Lambda_1+\Lambda_2\sqrt{\Gamma}<0,$ that is $\Gamma < \Big(\frac{\Lambda_1}{\Lambda_1}\Big)^2.$ Therefore, in the parametric region $$\mathcal{R}_2= \Big\{(\eta,\theta): 0<\eta<1<\theta,\,\,\frac{\sqrt{\Gamma}}{3}-3<\theta<1+\frac{\sqrt{\Gamma}}{3},\,\, \Gamma< \left(\frac{\Lambda_1}{\Lambda_2}\right)^2\Big\},$$we obtain two fold points in the interior of $\mathbb{R}_+^2.$ Similarly under the restriction \textbf{(P4)} there can be two sub-cases, either $\theta<1<\eta$ or $\theta<\eta<1.$ We consider the first sub-case and assume that for $\theta<\eta,$ $\Gamma>0.$ This gives us $$\theta<\eta<\frac{1}{8}(3\theta^2+2\theta+3).$$ Since $u_{\mathrm{min}} \in (0,1)$ such that $F({u_{\mathrm{min}}})<0,$ we must have $$\frac{\sqrt{\Gamma}}{3}-3<\theta<1,\,\, \text{and}\,\,\Gamma>\ \left(\frac{\Lambda_1}{\Lambda_2}\right)^2.$$ Therefore under the restriction \textbf{(P4)} we obtain exactly two fold points for $$\mathcal{R}_3= \Big\{(\eta,\theta): 0<\theta<1<\eta,\,\,\frac{\sqrt{\Gamma}}{3}-3<\theta<1,\,\, \Gamma> \left(\frac{\Lambda_1}{\Lambda_2}\right)^2\Big\}.$$
We now consider region $\mathcal{R}$ such as $\mathcal{R}=\mathcal{R}_1\cup \mathcal{R}_2\cup \mathcal{R}_3.$ Thus for $(\eta, \theta)\in \mathcal{R}$ the critical manifold is a $S$-shaped curve with exactly two fold points. Let the two fold points be $P(u_m,v_m)$ (local minimum of the critical manifold $M_{20}$) and $Q(u_M, v_M)$ (local maximum of the critical manifold $M_{20}$). Depending on the structure of the critical manifold $M_{20},$ that as $\lim_{u\rightarrow 0} \phi(u) \rightarrow \infty$ and $\lim_{u\rightarrow 1} \phi(u) \rightarrow 0$ we can infer that $u_m<u_M$. It consists of three branches, namely, $S_0^l$, $S_0^m$ and $S_0^r$, where
\begin{align*}
	S_0^l&=M_{20}\cap \left\{(u,v)\in \IR^2_+ \,\middle|\, 0<u<u_m\right\},\\
	S_0^m&=M_{20} \cap \left\{(u,v)\in \IR^2_+ \,\middle|\, u_m<u< u_M\right\},\\
	S_0^r&=M_{20} \cap \left\{(u,v)\in \IR^2_+ \,\middle|\, u_M<u< 1\right\}.
\end{align*}
From the geometry of the critical manifold $M_{20}$, we conclude that the sub-manifolds $S_0^{l}$ and $S_0^{r}$ are normally hyperbolic attracting whereas $S_0^{m}$ is normally hyperbolic repelling. When the non-trivial predator nullcline $u=\sqrt{\frac{\delta \eta}{1-\delta}}$ intersects the non-trivial prey nullcline $M_{20}$ on the left and right branch of $M_{20},$ that is on $S^l_0$ or $S^r_0$ then the equilibrium is stable. Whereas if $E_*$ lies on $S^m_0$, the equilibrium is unstable. We show the shape of the critical manifold along with the fold points in Fig. \ref{fig:phase_potraits}.

\begin{figure}[H]
    \centering
    \includegraphics[scale=0.5]{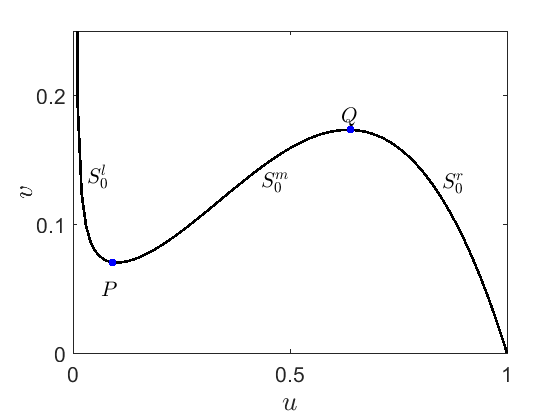}
     \caption{A representation of the critical manifold $M_{20}$ with the left ($S^l_0$), middle ($S^m_0$), and right ($S^r_0$) branch. The blue dots $P$ and $Q$ are the fold points of the critical manifold.}
    \label{fig:phase_potraits}
\end{figure}


\section{Singular Bifurcations}\label{sec4}

\noindent In this section, we focus our study on bifurcation structures in the singularly perturbed system \eqref{poly_sys}. Throughout the analysis, we assume that $(\eta, \theta)\in \mathcal{R}$ so that the critical manifold $M_{20}$ is an $S$-shaped curve having two non-degenerate fold points $P(u_m,v_m)$ and $Q(u_M, v_M)$. We will show that when the stable equilibrium $E_*$ on the left branch $S_0^l$ crosses the fold point $P$ for some threshold parameter value, the system \eqref{poly_sys} undergoes a Hopf bifurcation which is singular in the sense that at the Hopf bifurcation threshold, the eigenvalues of the Jacobian matrix become singular in the singular limit $\epsilon\to 0$. Subsequently, we will show that there will be a change in the criticality of the singular Hopf bifurcation and the system undergoes a codimension-2 Bautin bifurcation which we termed here as singular Bautin bifurcation as the corresponding Hopf bifurcation is singular.

\subsection{Singular Hopf Bifurcation}

We choose $\delta$ as the bifurcation parameter. We assume that for $\delta=\delta_*$, the equilibrium $E_*$ passes through the fold point $P(u_m, v_m)$ of the critical manifold $M_{20}$,  where $\delta_*$ is implicitly given by the equation $F(u_*)=0.$
Consequently, we have the following 
\begin{align}\label{canard condition 1}
  f(u_m,v_m,\delta_*)=0,\,\,g(u_m,v_m,\delta_*)=0,\,\,\text{and},\,\frac{\partial f}{\partial u}(u_m,v_m,\delta_*)=0.
\end{align}
We further assume that 
\begin{align}\label{canard condition 2}
  \frac{\partial f}{\partial v}(u_m,v_m,\delta_*)\ne0,\,\,\frac{\partial^2 f}{\partial u^2}(u_m,v_m,\delta_*)\ne0,\,\,\frac{\partial g}{\partial u}(u_m,v_m,\delta_*)\ne0 \,\,\, \frac{\partial g}{\partial \delta}(u_m,v_m,\delta_*)\ne0.
\end{align}
With the above assumption, the fold point $P$ is now the non-degenerate canard point or the singular contact point of the system. Using the transformation, $\Bar{u} = u-u_m,\,\,\Bar{v} = v-v_m,\,\,\lambda = \delta-\delta_*,$
we write the system \eqref{poly_sys} into the following form
\begin{equation}\label{sf_normal form}
\begin{aligned}
	\frac{d\bar{u}}{dt}&=a_{01}\bar{v}+\frac{a_{20}}{2}\bar{u}^2+a_{11}\bar{u}\bar{v}+\frac{a_{30}}{6}\bar{u}^3+\frac{a_{21}}{2}\bar{u}^2\bar{v}+\frac{a_{40}}{24}\bar{u}^4+\frac{a_{50}}{120}\bar{u}^5,\\
\frac{d\bar{v}}{dt}&=\epsilon(b_{10}\bar{u}+\frac{b_{20}}{2}\bar{u}^2+b_{11}\bar{u}\bar{v}-v_m(u_m^2+\eta)\lambda \\ &-2u_mv_m\lambda\bar{u}-(u_m^2+\eta)\lambda\bar{v} +\frac{b_{21}}{2}\bar{u}^2\bar{v}-v_m\lambda\bar{u}^2-\lambda\bar{u}^2\bar{v}),
\end{aligned}
\end{equation} 
where $\displaystyle a_{ij}=\left.\frac{\partial^{i+j}f}{\partial u^i\partial v^j}\right|_{(u_m,v_m,\delta_*)}$ and $\displaystyle b_{ij}=\left.\frac{\partial^{i+j}g}{\partial u^i\partial v^j}\right|_{(u_m,v_m,\delta_*)}$, i.e.,

\begin{equation*}
  \begin{aligned}
  & a_{10}=0,\,a_{01}=-u_m^2,\,a_{50}=-120,\,b_{01}=0,\\
  &a_{20}=\frac{4\delta_*\eta}{1-\delta_*}(3-3\theta-5u_m)+2\eta(1-3u_m-\theta)+6\theta u_m-2v_m,\,a_{11}=-2u_m,\\
      &a_{30}=24(1-\theta)u_m+6(\theta-\eta)-\frac{60\delta_*\eta}{1-\delta_*},\,a_{21}=-2,\,a_{40}=24(1-\theta)-120u_m,\\
       & b_{10}=2(1-\delta_*)u_mv_m,\,\,b_{20}=2(1-\delta_*)v_m,\,\,
        b_{11}=2(1-\delta_*)u_m,\,\,b_{21}=2(1-\delta_*).
    \end{aligned}
  \end{equation*}
One of the fundamental tools to study the flow of  \eqref{sf_normal form} near the canard point is to use the blow-up transformation
\begin{equation*}
\Phi: \mathbb{S}^3\times [0,\rho]\to\IR^4
\end{equation*}
given by $\bar{u}=r'u',~~\bar{v}=r'^2v',~~\lambda=r'\lambda'~~\epsilon=r'^2\epsilon',~ \text{where}~(u',v',\lambda',\epsilon')\in\mathbb{S}^3.$
The blow-up transformation, desingularizes the vector field  \eqref{sf_normal form}. One can then use the tools of dynamical systems to investigate the dynamics of the blown-up vector field in different charts of the manifold $\mathbb{S}^3\times [0,\rho]$ and connecting the results in different charts, the final results are derived by ``blow down" \cite{krupa2001relaxation,kuehn2015multiple}. The dynamics in the chart $K_2$ which describes a neighbourhood of the upper half-sphere are only necessary to study singular Hopf bifurcation. The blow-up transformation in the chart $K_2$ is given by 
\begin{align}\label{scaling}
\bar{u}=r_2u_2,~~~\bar{v}=r_2^2v_2,~~~\lambda=r_2\lambda_2~~~\epsilon=r_2^2.
\end{align} Using the above transformation and dividing by $r_2$, the system \eqref{sf_normal form} transforms into the following 
\begin{equation}\label{eq:5}
\begin{aligned}
\frac{du_2}{dt}&=a_{01}v_2+\frac{a_{20}}{2}u_2^2+\left(a_{11}u_2v_2+\frac{a_{30}}{6}u_2^3\right)r_2+\left(\frac{a_{21}}{2}u_2^2v_2+\frac{a_{40}}{24}u_2^4\right)r_2^2+\frac{a_{50}}{120}r_2^3u_2^5,\\
\frac{dv_2}{dt}&=b_{10}u_2-v_m(u_m^2+\eta)\lambda_2-2u_mv_mr_2\lambda_2u_2+\frac{b_{20}}{2}r_2u_2^2+b_{11}r_2^2 u_2v_2 \\ 
&-(u_m^2+\eta)r_2^2\lambda_2v_2-v_mr_2^2\lambda_2u_2^2+\frac{b_{21}}{2}r_2^3u_2^2v_2-r_2^4\lambda_2u_2^2v_2.
 \end{aligned}
\end{equation}
The system \eqref{eq:5} has an equilibrium $P_2(u_{2e}, v_{2e})$ with 
\begin{align*}
u_{2e}&=\frac{v_m(u_m^2+\eta)}{b_{10}}\lambda_2+\frac{1}{b_{10}^2}\left(2u_mv_m^2(u_m^2+\eta)-\frac{b_{20}}{2b_{10}}v_m^2(u_m^2+\eta)^2\right)r_2\lambda_2^2+\mathcal{O}(4),\\
v_{2e}&=-\frac{a_{20}}{2a_{01}b_{10}^2}v_m^2(u_m^2+\eta)^2\lambda_2^2+\mathcal{O}(4), 
\end{align*} 
where $\mathcal{O}(4)=\mathcal{O}(|\lambda_2 + r_2|^4)$. The Jacobian matrix at $P_2$ is given by
\begin{align*}
    J=\left[\begin{array}{cc} \alpha_{11} & \alpha_{12}\\
    \alpha_{21} & \alpha_{22}\end{array}\right]
\end{align*}
where \begin{align*}
\alpha_{11}&= \frac{a_{20}}{b_{10}}v_m(u_m^2+\eta)\lambda_2 + \frac{a_{20}}{b_{10}^2}\left(2u_mv_m^2(u_m^2+\eta)-\frac{b_{20}}{2b_{10}}v_m^2(u_m^2+\eta)^2\right)r_2\lambda_2^2\\
& \hspace{4cm} +\left(a_{30}-\frac{a_{20}a_{11}}{a_{01}}\right)\frac{v_m^2(u_m^2+\eta)^2}{2b_{10}^2}r_2\lambda_2^2+\mathcal{O}(4),\\
\alpha_{12}&=a_{01}+\frac{a_{11}}{b_{10}}v_m(u_m^2+\eta)r_2\lambda_2+\mathcal{O}(4),\\
&\alpha_{21}=b_{10}-2u_mv_mr_2\lambda_2+\frac{b_{20}}{b_{10}}v_m(u_m^2+\eta)r_2\lambda_2+\mathcal{O}(4),\\
\alpha_{22}&=\frac{b_{11}}{b_{10}}v_m(u_m^2+\eta)r_2^2\lambda_2-(u_m^2+\eta)r_2^2\lambda_2+\mathcal{O}(4).
\end{align*}
Now, using the transformation $x=u_2-u_{2e}$, $y=v_2-v_{2e}$ and expanding into Taylor's series, we write the system \eqref{eq:5} as
\begin{subequations}\label{eq:6}
\begin{align}
\frac{dx}{dt}&=a_{10}'x+a_{01}'y+\frac{a_{20}'}{2}x^2+a_{11}'xy+\frac{a_{30}'}{6}x^3+\frac{a_{21}'}{2}x^2y+\frac{a_{40}'}{24}x^4+\frac{a_{50}'}{120}x^5,\\
\frac{dy}{dt}&=b_{10}'x+b_{01}'y+\frac{b_{20}'}{2}x^2+b_{11}'xy+\frac{b_{21}'}{2}x^2y,
\end{align}
\end{subequations}
where,
\begin{align*}
a_{10}'&=\alpha_{11},\,a_{01}'=\alpha_{12},\,b_{10}'=\alpha_{21},\,b_{01}'=\alpha_{22},\,
a_{20}'=a_{20}+\frac{a_{30}}{b_{10}}v_m(u_m^2+\eta)r_2\lambda_2+\mathcal{O}(4),\\
a_{11}'&=a_{11}r_2+\frac{a_{21}}{b_{10}}v_m(u_m^2+\eta)r_2\lambda_2+\mathcal{O}(4),\\
a_{30}'&=a_{30}r_2+\frac{a_{40}}{b_{10}}v_m(u_m^2+\eta)r_2\lambda_2+\mathcal{O}(4),\\
a_{21}'&=a_{21}r_2^2,\,a_{40}'=a_{40}r_2^2+\mathcal{O}(4),\,a_{50}'=a_{50}r_2^3,\,
b_{20}'=b_{20}r_2-2v_mr_2\lambda_2+\mathcal{O}(4),\\
b_{11}'&=b_{11}r_2^2+\mathcal{O}(4),\,\text{and}\, b_{21}'=b_{21}r_{2}^3+\mathcal{O}(4).
\end{align*}

\noindent For the Hopf bifurcation, Trace $J=0$ gives rise to $\lambda_2=0$ and in terms of original parameters this transforms into
$$\delta=\delta_H=\delta_*.$$ Thus, we observe that the Hopf bifurcation threshold does not depend on $\epsilon$. For $|\delta-\delta_H|$ small, the complex conjugate eigenvalues of the Jacobian matrix $J$ are given by $\mu$ and $\bar{\mu}$, where
\begin{align*}
\mu=\alpha+i\beta,\,\,\,\alpha=\frac{\alpha_{11}+\alpha_{22}}{2},\,\,\,\beta=\frac{1}{2}\sqrt{4(\alpha_{11}\alpha_{22}-\alpha_{12}\alpha_{21})-(\alpha_{11}+\alpha_{22})^2}.
\end{align*}
Consequently, at $\delta=\delta_H$, we have 
$\alpha(\delta_H)= 0,\,\,\beta(\delta_H)=\beta_0+\mathcal{O}(4)$
where $\beta_0=\sqrt{-a_{01}b_{10}}$. We also have the following
\begin{align*}
\frac{d}{d\delta}\left(Re(\mu(\delta))\right)\bigg|_{\delta=\delta_H}&\neq 0.
\end{align*}

\noindent Let $q=\left(q_1,q_2\right)^T\in \mathbb{C}$ be an eigenvector of $J$ corresponding to the eigenvalue $\mu$ and $p=\left(p_1, p_2\right)^T$ be an eigenvector of $J^T$ corresponding to $\bar{\mu}$ such that 
$\langle p, q\rangle= \bar{p}_1q_1+\bar{p}_2q_2=1$. We then have

\begin{align*}
\left[ \begin {array}{c} q_1 \\ q_2
\end {array} \right] = \left[ \begin {array}{c} \alpha_{12} \\ \mu-\alpha_{11}
\end {array} \right],~~~\left[ \begin {array}{c} p_1 \\ p_2
\end {array} \right] =\frac{1}{2\alpha_{12}\bar{\mu}-\alpha_{12}(\alpha_{11}+\alpha_{22})}\left[ \begin {array}{c} \bar{\mu}-\alpha_{22}\\ \alpha_{12}
\end {array} \right]
\end{align*}

With the help of the transformation $z =\bar{p}_1x+\bar{p}_2y, \,\,\,z\in\mathbb{C},$ the system \eqref{eq:6} can be written as 
\begin{align}\label{eq:7}
    \dot{z}=\mu z+\sum_{k+l=2}^5\frac{g_{kl}}{k!l!}z^k\bar{z}^l,
\end{align}
where the expressions of $g_{kl}$'s are given by 
\begin{subequations}
\begin{align*}
    g_{11}&=\bar{p}_1(a_{20}'q_1\bar{q}_1+a_{11}'(q_1\bar{q}_2+q_2\bar{q}_1)+\bar{p}_2(b_{20}'q_1\bar{q}_1+b_{11}'(q_1\bar{q}_2+q_2\bar{q}_1),\,g_{05}=a_{50}'\bar{p}_1\bar{q}_1^5,\\
    g_{20}&=\bar{p}_1(a_{20}'q_1^2+2a_{11}'q_1q_2)+\bar{p}_2(b_{20}'q_1^2+2b_{11}'q_1q_2),\\
    g_{02}&=\bar{p}_1(a_{20}'\bar{q}_1^2+2a_{11}'\bar{q}_1\bar{q}_2)+\bar{p}_2(b_{20}'\bar{q}_1^2+2b_{11}'\bar{q}_1\bar{q}_2),\\
    g_{21}&=\bar{p}_1(a_{30}'q_1^2\bar{q}_1+a_{21}'(q_1^2\bar{q}_2+2q_1\bar{q}_1q_2))+b_{21}'\bar{p}_2 (q_1^2\bar{q}_2+2q_1\bar{q}_1q_2),\\
    g_{22}&=a_{40}'\bar{p}_1q_1^2\bar{q}_1^2,\,\,g_{23}=a_{50}'\bar{p}_1q_1^2\bar{q}_1^3,\\
    g_{12}&=\bar{p}_1(a_{30}'q_1\bar{q}_1^2+a_{21}'(\bar{q}_1^2q_2+2q_1\bar{q}_1\bar{q}_2))+b_{21}'\bar{p_2} (\bar{q}_1^2q_2+2q_1\bar{q}_1\bar{q}_2),\\
    g_{13}&=a_{40}'\bar{p}_1q_1\bar{q}_1^3,\,\,g_{14}=a_{50}'\bar{p}_1q_1\bar{q}_1^4,\,\,g_{30}=\bar{p}_1(a_{30}'q_1^3+3a_{21}'q_1^2q_2)+3b_{21}'\bar{p}_2q_1^2q_2,\\
    g_{03}&=\bar{p}_1(a_{30}'\bar{q_1}^3+3a_{21}'\bar{q}_1^2\bar{q}_2)+3b_{21}'\bar{p}_2\bar{q}_1^2\bar{q}_2,\,\, g_{40}=a_{40}'\bar{p}_1q_1^4,\,\,g_{04}=a_{40}'\bar{p}_1\bar{q}_1^4,\\
    g_{50}&=a_{50}'\bar{p}_1q_1^5,\,\,g_{41}=a_{50}'\bar{p}_1q_1^4\bar{q}_1,\,\, g_{32}=a_{50}'\bar{p}_1q_1^3\bar{q}_1^2,\,g_{31}=a_{40}'\bar{p}_1q_1^3\bar{q}_1.
\end{align*}
\end{subequations}

Using the following locally invertible parameter-dependent change of complex coordinate,
\begin{align}
    z=w+\frac{h_{20}}{2}w^2+h_{11}w\bar{w}+\frac{h_{02}}{2}\bar{w}^2+\frac{h_{30}}{6}w^3+\frac{h_{12}}{2}w\bar{w}^2+\frac{h_{03}}{6}\bar{w}^3,
\end{align}
the system \eqref{eq:7} can be written in a neighbourhood of $\delta=\delta_H$ as 
\begin{align}\label{eq:8}
\dot{w}=\mu w+c_1w^2\bar{w}+\mathcal{O}\left(|w|^4\right),
\end{align}
where  
\begin{equation}\label{eq:9}
\begin{aligned}
    c_1&=\frac{g_{20}g_{11}\left(2\mu+\bar{\mu}\right)}{2|\mu|^2}+\frac{|g_{11}|^2}{\mu}+\frac{|g_{02}|^2}{2\left(2\mu-\bar{\mu}\right)}+\frac{g_{21}}{2},\\
    &h_{20}=\frac{g_{20}}{\mu},\,\, h_{11}=\frac{g_{11}}{\bar{\mu}},\,\,\, h_{02}=\frac{g_{02}}{2\bar{\mu}-\mu},\\
    h_{30}&=\frac{3}{\mu}\left(\frac{g_{20}h_{20}}{2}+\frac{g_{11}\overline{h}_{02}}{2}+\frac{g_{30}}{6}\right),\,\,
    h_{12}=\frac{1}{2\bar{\mu}}\left(g_{20}h_{02}+2g_{11}h_{11}+g_{11}\overline{h}_{20}+g_{12}\right),
    \\
    h_{03}&=\frac{1}{3\bar{\mu}-\mu}\left(g_{03}+3g_{11}h_{02}+3g_{02}\overline{h}_{20}\right),
\end{aligned}
\end{equation}
and the first Lyapunov coefficient using \cite{kuznetsov1998elements} is given by 
\begin{align}
    L_1(\delta_H)&=\frac{\Re \left(c_1\right)}{\beta}\bigg|_{\delta=\delta_H}=\frac{1}{2\beta^2}\Re\left(i g_{20}g_{11}+\beta_0 g_{21}\right)\bigg|_{\delta=\delta_H},
\end{align}
where $\Re(\cdot)$ stands for the real part of $(\cdot)$. 

At $\delta=\delta_H$,
\begin{align*}
\left[ \begin {array}{c} q_1(\delta_H) \\ q_2(\delta_H) 
\end {array} \right] = \left[ \begin {array}{c} a_{01} \\ i\beta_0
\end {array} \right],~~~\left[ \begin {array}{c} p_1(\delta_H) \\ p_2(\delta_H) 
\end {array} \right] =\left[ \begin {array}{c} \frac{1}{2a_{01}} \\ \frac{i}{2\beta_0}
\end {array} \right],
\end{align*}
and thus, the first Lyapunov coefficient is 
\begin{align}\label{Eq:L1}
 L_1(\sqrt{\epsilon})&=-\frac {a_{01}}{4\beta_0b_{10}}A\sqrt {\epsilon},
\end{align}
where 
\begin{align}\label{criticality}
  A&= a_{01}a_{20}b_{20}-a_{01}a_{30}b_{10}+a_{11}a_{20}b_{10}\textcolor{red}{.}
\end{align}

Thus, we see that all the conditions of the Hopf bifurcation theorem are satisfied and consequently, the system \eqref{poly_sys} undergoes a Hopf bifurcation at $\delta=\delta_H(\sqrt{\epsilon})=\delta_*.$ It also follows that for $0 < \epsilon \ll 1$, the leading order coefficient of $L_1$ i.e., $A$ determines the criticality of the singular Hopf bifurcation. The criticality changes if $A$ changes its sign from negative to positive and consequently, the singular Hopf bifurcation is degenerate if $A=0$. We now summarize the above-mentioned results by the following theorem.

\begin{theorem}\label{sf_singular_Hopf}
Let\textcolor{black}{,} $(\theta,\eta)\in \mathcal{R}$ and the conditions \eqref{canard condition 1} and \eqref{canard condition 2} hold. Then $\exists$ $\epsilon_0>0$ and $\delta_0>0$ such that for $0<\epsilon<\epsilon_0$ and $|\delta-\delta_*|<\delta_0$, the system \eqref{poly_sys} has an equilibrium point $P_2$ in a neighbourhood of the fold point $P$ which converges to $P$ as $(\epsilon, \delta)\to (0,\delta_*)$. The system \eqref{poly_sys} undergoes a singular Hopf bifurcation at 
\begin{align}
\delta_H(\sqrt{\epsilon})=\delta_*+\mathcal{O}(\epsilon^{\frac{5}{2}})
\end{align}

The Hopf bifurcation is non-degenerate when $A\neq 0$. It is supercritical if $A<0$ and sub-critical if $A>0$ where $A$ is given by \eqref{criticality}.
\end{theorem}

\subsection{Singular Bautin Bifurcation}

We observe that for $\theta=\theta_B$, $A=0$ where $\theta_B$ is given by 
\begin{align}
 \theta=\theta_B=-\frac{{u_m} \left(\left(-1+\delta^* \right) {u_m}^{3}+\left(-5 \delta^* +5\right) {u_m}^{2}-\eta  \left(-1+\delta^* \right) {u_m} -6 \delta^*  \eta \right)}{\left(5 \delta^* -5\right) {u_m}^{3}+\left(-1+\delta^* \right) {u_m}^{2}+6 \delta^*  \eta  {u_m} -\eta  \left(-1+\delta^* \right)}.   
\end{align}
Consequently, it follows from theorem \ref{sf_singular_Hopf} that the Hopf bifurcation is degenerate i.e., there is a change in the criticality of the singular Hopf bifurcation as $\theta$ passes through $\theta=\theta_B$. We now proceed to show that in such a case, the system undergoes a codimension-2 Bautin bifurcation by computing the second Lyapunov coefficient, which determines the criticality of the Bautin bifurcation.

\noindent To derive the normal form for Bautin bifurcation, we repeat the same process as carried out in the previous section, i.e., we use the following locally invertible parameter-dependent change of complex coordinate
\begin{align*}
    z&=w+\frac{h_{20}}{2}w^2+h_{11}w\bar{w}+\frac{h_{02}}{2}\bar{w}^2+\frac{h_{30}}{6}w^3+\frac{h_{12}}{2}w\bar{w}^2+\frac{h_{03}}{6}\bar{w}^3+\frac{h_{40}}{24}w^4\\
    & \hspace{1cm}+\frac{h_{31}}{6}w^3\bar{w} +\frac{h_{22}}{4}w^2\bar{w}^2+\frac{h_{13}}{6}w\bar{w}^3+\frac{h_{04}}{24}\bar{w}^4+\frac{h_{50}}{120}w^5\\
    & \hspace{2cm}+\frac{h_{41}}{24}w^4\bar{w}+\frac{h_{23}}{12}w^2\bar{w}^3+\frac{h_{14}}{24}w\bar{w}^4+\frac{h_{05}}{120}\bar{w}^5,
\end{align*}
to eliminate all the quadratic terms and bi-quadratic terms of the equation \eqref{eq:7} and left with only resonant cubic and fifth-degree terms. The equation \eqref{eq:7} then transforms into, 
\begin{align}
    \dot{w}=\mu w+c_1w^2\bar{w}+c_2w^3\bar{w}^2+\mathcal{O}\left(|w|^6\right),
\end{align} where the expressions for $c_1$, $h_{ij}, 2\leq i+j\leq 3$ are given by \eqref{eq:9} and the expressions for $c_2$, $h_{ij}, i+j=4$ are given below. We have not mentioned the expressions for $h_{ij}, i+j=5$ here, as these coefficients are no longer required to compute the second Lyapunov coefficient.

\begin{align*}
    c_2 &=\frac{1}{2}g_{20}\left(\frac{h_{20}h_{12}}{2}+\frac{h_{22}}{2}+\frac{h_{30}h_{02}}{6}\right)\\
    & \hspace{2cm}+g_{11}\left(\frac{\bar{h}_{22}}{4}+\frac{h_{11}\bar{h}_{12}}{2}+\frac{h_{02}\bar{h}_{03}}{12}+\frac{h_{30}\bar{h}_{20}}{12}+\frac{h_{12}\bar{h}_{02}}{4}+\frac{h_{31}}{3}\right)\\
    & \hspace{2cm} +\frac{g_{02}}{2}\left(\frac{\bar{h}_{20}\bar{h}_{03}}{6}+
    \bar{h}_{11}\bar{h}_{12}+\frac{\bar{h}_{13}}{3}\right)+\frac{1}{6}g_{30}\left(3{h_{11}}^2+\frac{3h_{12}}{2}+\frac{3h_{20}h_{02}}{2}\right)\\
    & \hspace{2cm} +\frac{1}{2}g_{21}\left(\frac{|h_{20}|^2}{2}+2|h_{11}|^2+h_{20}h_{11}+\frac{|h_{02}|^2}{2}\right)\\
    & \hspace{2cm}+\frac{1}{2}g_{12}\left(\frac{h_{30}}{6}+\bar{h}_{11}^2+\bar{h}_{11}h_{20}+h_{11}\bar{h}_{02}+\bar{h}_{12}\right)\\
    & \hspace{2cm} \hspace{.2cm} +\frac{1}{6}g_{03}\left(\frac{\bar{h}_{03}}{2}+3\bar{h}_{11}\bar{h}_{02}\right)+\frac{g_{40}h_{02}}{12}+\frac{g_{31}}{6}\left(3h_{11}+\frac{\bar{h}_{20}}{2}\right)\\
    & \hspace{2cm}+\frac{g_{22}}{4}\left(2\bar{h}_{11}+h_{20}\right)+\frac{g_{13}\bar{h}_{02}}{4}+\frac{g_{32}}{12},
\end{align*}

\begin{align*}
    h_{40}&=\frac{8}{\mu}\left[\frac{g_{20}}{2}\left(\frac{h_{20}^2}{4}+
    \frac{h_{30}}{3}\right)+g_{11}\left(\frac{\bar{h}_{03}}{6}+\frac{h_{20}\bar{h}_{02}}{4}\right)+\frac{g_{02}\bar{h}_{02}^2}{8}
    +\frac{g_{30}h_{20}}{4}+\frac{g_{21}\bar{h}_{02}}{4}+\frac{g_{40}}{24}\right],\\
    h_{31}&=\frac{6}{2\mu+\bar{\mu}}\left[\frac{g_{20}}{2}\left(h_{21}+h_{20}h_{11}\right)+g_{11}\left(\frac{\bar{h}_{12}}{2}+\frac{h_{20}\bar{h}_{11}}{2}+\frac{h_{11}\bar{h}_{02}}{2}+\frac{h_{30}}{6}\right)\right.\\
    &\hspace{.2cm}+\frac{g_{02}}{2}\left(\frac{\bar{h}_{03}}{3}+\bar{h}_{11}\bar{h}_{02}\right)
      +\left.\frac{g_{30}h_{11}}{2}+\frac{g_{21}}{2}\left(h_{20}+\bar{h}_{11}\right)+\frac{g_{12}\bar{h}_{02}}{2}+\frac{g_{31}}{6}-c_1h_{20}\right],\\
    h_{22}&=\frac{4}{\mu+2\bar{\mu}}\left[\frac{g_{20}}{2}\left(h_{11}^2+h_{12}+\frac{h_{20}h_{02}}{2}\right)+g_{11}\left(\frac{|h_{20}|^2}{4}+|h_{11}|^2+\frac{|h_{02}|^2}{4}\right)
    + \frac{g_{21}}{2}\left(2h_{11}+\frac{\bar{h}_{20}}{2}\right)\right.\\
        & \hspace{.2cm} +\left.\frac{g_{02}}{2}\left(\frac{\bar{h}_{20}\bar{h}_{02}}{2}+\bar{h}_{11}^2+\bar{h}_{12}\right) +\frac{g_{30}h_{02}}{4}
        +\frac{g_{12}}{2}\left(2\bar{h}_{11}+\frac{h_{20}}{2}\right)+\frac{g_{03}\bar{h}_{02}}{4}+\frac{g_{22}}{4}-2h_{11}\Re(c_1)\right],\\
    h_{13}&=\frac{2}{\bar{\mu}}\left[\frac{g_{20}}{2}\left(h_{11}h_{02}+\frac{h_{03}}{3}\right)+g_{11}\left(\frac{\bar{h}_{30}}{6}+\frac{h_{12}}{2}+\frac{h_{11}\bar{h}_{20}}{2}+\frac{\bar{h}_{11}h_{02}}{2}\right)+\frac{g_{02}}{2}\left(\bar{h}_{21}+\bar{h}_{11}\bar{h}_{20}\right)\right.\\
    & \hspace{.2cm} +\left.\frac{g_{21}h_{02}}{2}+\frac{g_{12}\left(h_{11}+\bar{h}_{20}\right)}{2}+\frac{g_{03}\bar{h}_{11}}{2}+\frac{g_{13}}{6}-\bar{c}_1h_{02}\right],\\
    h_{04}&=\frac{24}{4\bar{\mu}-\mu}\left[\frac{g_{20}h_{02}^2}{8}+g_{11}\left(\frac{h_{03}}{6}+\frac{h_{02}\bar{h}_{20}}{4}\right)+\frac{g_{02}}{2}\left(\frac{\bar{h}_{20}^2}{4}+\frac{\bar{h}_{30}}{3}\right)
    +\frac{g_{12}h_{02}}{4}+\frac{g_{03}\bar{h}_{20}}{4}+\frac{g_{04}}{24}\right].
\end{align*}

\noindent Assuming,
\begin{align}\label{Bautin_1}
L_1(\delta=\delta_H, \theta=\theta_B)=\alpha(\delta=\delta_H, \theta=\theta_B)=0,
\end{align}
the second Lyapunov coefficient at $(\delta=\delta_H, \theta=\theta_B)$ is given by 
\begin{align}\label{second Lyapunov coeffiecient_1}
L_2(\delta=\delta_H, \theta=\theta_B)=\frac{\Re \left(c_2\right)}{\beta}\bigg|_{(\delta=\delta_H, \theta=\theta_B)}.
\end{align}

\noindent An expression for $L_2(\delta=\delta_H, \theta=\theta_B)$ as mentioned in \cite{kuznetsov1998elements} in compact form is given by the following,
\begin{align*}
&12L_2(\delta=\delta_H,\theta=\theta_H)\\
& \hspace{1.cm}=\frac{\Re(g_{32})}{\beta}+\frac{1}{\beta^2}\Im\left[g_{20}\bar{g}_{31}-g_{11}\left(4g_{31}+3\bar{g}_{22}\right)-\frac{1}{3}g_{02}\left(g_{40}+
\bar{g}_{13}\right)-g_{30}g_{12}\right]\\
& \hspace{1cm} +\frac{1}{\beta^3}\left\{\Re\left[g_{20}\left(\bar{g}_{11}(3g_{12}-\overline{g_{30}})+
g_{02}(\bar{g}_{12}-\frac{g_{30}}{3})+\frac{1}{3}\bar{g}_{02}g_{03}\right)\right.\right.\\
& \hspace{1cm} \left.\left. +~ g_{11}\left(\bar{g}_{02}(\frac{5}{3}\bar{g}_{30}+3g_{12})+\frac{1}{3}g_{02}\bar{g}_{03}-4g_{11}g_{30}\right)\right]+3\Im\left(g_{20}g_{11}\right)\Im(g_{21})\right\}\\
& \hspace{1cm} +\frac{1}{\beta^4}\left\{\Im\left[g_{11}\bar{g}_{02}(\bar{g}_{20}^2-
3\bar{g}_{20}g_{11}-4g_{11}^2)\right]+\Im(g_{20}g_{11})\left[3\Re(g_{20}g_{11})-2|g_{02}|^2\right]\right\},
\end{align*}
where all the expressions of $g_{kl}$ and $\beta$ are evaluated at the point $(\delta=\delta_H, \theta=\theta_B)$, and  $\Im(\cdot)$ stands for the imaginary part of $(\cdot)$. After some computation, we have
\begin{align}\label{second_Lyapunov_coefficient}
  L_2(\sqrt{\epsilon})&=  B\epsilon^{\frac{3}{2}},  
\end{align}
with
\begin{equation}\label{second Lyapunov coefficient_2}
\begin{aligned}
B=&\frac{a_{01}^{4}b_{10}}{144 \beta_0^{7}} \Big(18a_{11}^{3}a_{20} b_{10}^{2}+\left(27a_{20}b_{20} -2a_{30}b_{10}\right)a_{01} a_{11}^{2}b_{10}+\left(7a_{30} b_{10} b_{20}-27 b_{20}^{2} a_{20}  -10 b_{10}^{2} a_{40} \right) a_{01}^{2} a_{11}\\
   &+2\left(10 a_{40} b_{20} -3 a_{50} b_{10} \right) a_{01}^{3} b_{10} +\left(48 b_{20}a_{20}^{2}b_{11} +36a_{21}a_{30}b_{10}^{2}-30 a_{20}^{2} b_{21} b_{10} \right)a_{01}^{2}\Big)
\end{aligned}
\end{equation}

We have $A=0$ for $\theta=\theta_B$ and consequently, for fixed $0 < \epsilon \ll 1$ the criticality of the second Lyapunov coefficient $L_2$ at the point $(\delta=\delta_H, \theta=\theta_B)$ is determined by $B$. 

\noindent Thus, assuming the following conditions, 
\begin{equation}\label{Bautin_2}
\begin{aligned}
& B(\delta=\delta_H, \theta=\theta_B) \neq 0,\,\,\text{and}\,\,
{\rm Det} \left[ \begin {array}{cc} \frac{\partial \alpha}{\partial\delta} & \frac{\partial \alpha}{\partial\theta}\\ \noalign{\medskip} 
\frac{\partial L_1}{\partial\delta} & \frac{\partial L_1}{\partial\theta}
\end {array} \right]_{\delta=\delta_H,\ \theta=\theta_B}\neq 0,
\end{aligned}
\end{equation}
we observe that all the conditions of the Bautin bifurcation hold \cite{kuznetsov1998elements} and the system undergoes a codimension-2 Bautin bifurcation at $(\delta, \theta)=(\delta_H, \theta_B)$. We now summarize the above results by the following theorem.

\begin{theorem}\label{sf_singular_Bautin}
Assume $(\theta,\eta)\in \mathcal{R}$ and the conditions \eqref{canard condition 1}, \eqref{canard condition 2}, \eqref{Bautin_1} and \eqref{Bautin_2} holds. Then $\exists$ $\epsilon_0>0$, $\delta_0>0$ and $\theta_0>0$ such that for $0<\epsilon<\epsilon_0$ and $|\delta-\delta_*|<\delta_0$, $|\theta-\theta_*|<\theta_0$ the system \eqref{poly_sys} has an equilibrium point $P_2$ in a neighbourhood of the fold point $P$ which converges to $P$ as $(\epsilon, \delta, \theta)\to (0,\delta_*, \theta_*)$.
The system \eqref{poly_sys} undergoes a codimension-2 Bautin bifurcation at $(\delta, \theta)=(\delta_H, \theta_B)$.
\end{theorem} 

The Bautin bifurcation shows that for fixed $\theta$, $\theta<\theta_B$, the system undergoes a subcritical singular Hopf bifurcation at $\delta=\delta_H$. Subsequently, under the variation of parameter $\theta$ in a neighbourhood of $\theta=\theta_B$, the singular Hopf bifurcation is then accompanied by a saddle-node bifurcation of canard cycles showing the appearance of unstable and stable canard cycles, coalescent and disappearance of the cycles. 

\section{Canard Explosion and Relaxation Oscillation}\label{sec5}

We have shown in the previous section that the model system \eqref{poly_sys} undergoes a singular Hopf bifurcation at $\delta=\delta_H$ assuming the conditions \eqref{canard condition 1}, \eqref{canard condition 2} and $A\neq 0$. Now, the small $\mathcal{O}(\sqrt{\epsilon})$ amplitude family of limit cycles generated through the singular Hopf bifurcation at $\delta=\delta_H$ changes abruptly within an exponentially small parameter interval of the bifurcation parameter $\delta$. This phenomenon in literature is known as ``canard explosion" which connects the small amplitude limit cycles generated through the singular Hopf bifurcation to large relaxation cycles in a continuous fashion. Hence, based on the theorems 3.2 and 3.3 of \cite{krupa2001relaxation} and theorem 8.4.3 of \cite{kuehn2015multiple}, we have the following results for the existence of maximal canard and canard cycles for the system \eqref{poly_sys}.

\begin{theorem}\label{sf_maximal canard}
Let $(\theta,\eta)\in \mathcal{R}$ and the conditions \eqref{canard condition 1} and \eqref{canard condition 2} hold. Then $\exists$ $\epsilon_0>0$ and $\delta_0>0$ such that for $0<\epsilon<\epsilon_0$ and $|\delta-\delta_*|<\delta_0$, the system \eqref{poly_sys} has an equilibrium point $P_2$ in a neighbourhood of the fold point $P$ which converges to $P$ as $(\epsilon, \delta)\to (0,\delta_*)$. The system \eqref{poly_sys} admits a maximal canard at
\begin{align}\label{canard curve}
\delta_C=\delta_*- \left(\frac{b_{10}}{2a_{20}^3v_m(u_m^2+\eta)}A\right)\epsilon +\mathcal{O}(\epsilon^2).
\end{align}
\end{theorem}

\begin{theorem}\label{sf_canard explosion}
Let $(\theta,\eta)\in \mathcal{R}$, the conditions \eqref{canard condition 1} and \eqref{canard condition 2} hold and $\delta_C$ is the maximal canard value. Then there exists a smooth parameterized family of canard cycles growing from $\mathcal{O}(\sqrt{\epsilon})$ amplitude to a relaxation oscillation within an exponentially small parameter interval. However, if $A<0,$ then the singular Hopf bifurcation is supercritical and the family of canard cycles uniformly close to the canard point is stable. Whereas, if $A>0,$ then the family of canard cycles is unstable and there exists a unique parameter value $\delta_{\mathrm{SNL}}$ where the cycles undergo a saddle-node bifurcation of limit cycles. 
\end{theorem}

\noindent Canard explosion and relaxation oscillation are both considered to be global phenomena that can be observed in systems with multiple timescales, as they characterize transitions from a small amplitude to a large amplitude canard cycle. We next prove the existence of relaxation oscillation in the system \eqref{poly_sys}.

\begin{theorem}
Assume that $(\eta,\theta) \in \mathcal{R}$ and $(u_m,v_m)$ and $(u_M,v_M)$ be the two fold points of the critical manifold $M_{20}$ where $u_m<u_M.$ Let $u_m<u_*<u_M,$ then, for $\epsilon>0$ sufficiently small, the system has a unique relaxation oscillation $\gamma_{\epsilon}$.
\end{theorem}

\begin{proof}
Let $(u_l,v_M)$ be the point of intersection of the horizontal line $v=v_M$ with the critical manifold $S^l_0$ and $(u_r,v_m)$ be the point of intersection of the line $v=v_m$ with the critical manifold $S^r_0.$ We then define the singular trajectory $\gamma_0$ as a union of two alternative slow and fast flows. Let us define the  horizontal segments joining $(u_m,v_m)$ to $(u_r,v_m)$ as $l_1$ and $(u_M,v_M)$ to $(u_l,v_M)$ as $l_2.$ Then the fast flows are along the line segments $l_1$ and $l_2.$
Let us define the slow flow along the sub-manifold $S^l_0$ from $(u_l,v_M)$ to $(u_m,v_m)$ as $c_l$ and along $S^r_0$ from $(u_r,v_m)$ to $(u_M,v_M)$ as $c_r.$
We then define the singular trajectory as $$\gamma_0 = l_1 \cup c_r\cup l_2 \cup c_l.$$
Let the unique coexistence equilibrium $E_*=(u_*,v_*)$ lies on the normally hyperbolic repelling branch $S^m_0.$ From the slow-fast dynamics of the system \eqref{temp_model}, the equilibrium $E_*$ is unstable. The two fold points $(u_m,v_m)$ and $(u_M,v_M)$ are the jump points of the system. We now consider a small horizontal section  $\Delta$ transversal to $S^l_0,$ where
$$\Delta = \{(u,v_0): u \in [a,a+\rho],\,\, 0<a<u_m,\,\, v_m<v_0<v_M,\,\, \rho>\epsilon>0\},$$
and track two trajectories $\xi_{\epsilon}^1,\, \xi_{\epsilon}^2$ starting on $\Delta.$ From Fenichel's theorem \cite{fenichel1979geometric}, for $\epsilon>0,$ the normally hyperbolic sub-manifolds $S^l_0$ and $S^r_0$ perturbs to $S^l_{\epsilon}$ and $S^r_{\epsilon}$ respectively. Therefore, for $\epsilon>0$ the trajectories starting at $\Delta$ get attracted towards $S^l_{\epsilon}$ with an exponential rate. From the vicinity of the fold point $(u_m,v_m),$ it then jumps to the other attracting manifold $S^r_{\epsilon},$ and follows until the vicinity of another jump point $(u_M,v_M).$ The trajectories then jump and get exponentially attracted towards $S^l_{\epsilon}$ and return to $\Delta.$ We then define a return map $\Pi: \Delta \rightarrow \Delta.$ From \cite{krupa2001extending}, it shows that the map $\Pi$ is a contraction map with an exponential contracting rate. Therefore, from the contraction mapping theorem, there exists a unique fixed point. This fixed point is attracting, which gives rise to a unique limit cycle $\gamma_{\epsilon}.$ By \cite{fenichel1979geometric,krupa2001extending} we can conclude that $\gamma_{\epsilon} \rightarrow \gamma_0$ as $\epsilon\rightarrow0.$
\end{proof}

\begin{figure}[ht!]
    \centering
    \includegraphics[scale=0.5]{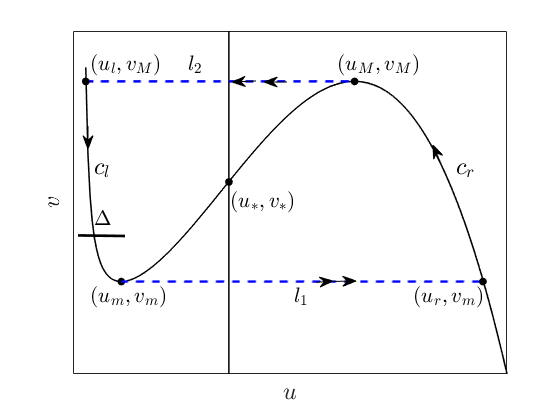}
    \caption{A schematic diagram of the singular trajectory $\gamma_0.$}
    \label{fig:my_label}
\end{figure}

\begin{theorem}\label{unique_RO}
Assume that $(\eta,\theta) \in \mathcal{R}$ and $P(u_m,v_m), Q(u_M,v_M)$ be the two fold points of the critical manifold $M_{20}$. Let $u_*>u_M,$ then, for $\epsilon>0$ sufficiently small, the coexistence equilibrium $E_*(u_*, v_*)$ is globally stable.
\end{theorem}

\begin{proof} Let the predator nullcline of the system \eqref{poly_sys} $u=\sqrt{\frac{\delta \eta}{1-\delta}}$ coincides with the sub-manifold $S^r_0$ of the critical manifold $M_{20}$ at the coexistence equilibrium $E_*.$ From the geometry of the $S^r_0,$ we can infer that $E_*$ is locally stable. Let us take a vertical line $u=u_M$ which divides the interior of the first quadrant $\IR^2_+$ into two regions $D_1$ and $D_2$, where 
\begin{align*}
    D_1=\{(u,v):0<u\leq u_M, v>0\}\,\,\text{and}\,\,
    D_2=\{(u,v):u>u_M, v>0\}.
\end{align*}
In the region $D_2$, we define the Dulac function $D:D_2\to\IR$ by $D(u,v)=\frac{1}{u^2v}$. It then follows from the nature of the critical manifold $M_{20}$ that 
\begin{align*}
 \frac{\partial (fD)}{\partial u} +\frac{\partial(\epsilon gD)}{\partial v}=
 \phi'(u)<0,\,\, \forall\,\,u,v\in D_2.  
\end{align*}
Hence, by Dulac criterion, there exists no periodic orbit in $D_2$. Thus, any trajectory starting in the region $D_2$ converges to $E_*.$ For global stability we now track the trajectories starting in $D_1.$ We consider two trajectories: $\xi_1$ starting above the critical manifold $M_{20}$ and $\xi_2$ starting below $M_{20}.$ Since $S^l_0$ is a normally hyperbolic attracting sub-manifold, $\xi_1$ gets attracted to the vicinity of $S^l_0$ till it encounters the fold point $P.$ Note that both the fold points $P, Q$ are jump points and the fast flow moves away from the critical manifold $M_{20}$. It then gets attracted to the vicinity of $S^r_0$ which lies in the region $D_2.$ Thus it converges to $E_*.$ However, the trajectory $\xi_2$ starting below $M_{20}$ and in between two fold points $P$ and $Q$ gets attracted in the region $D_1.$ Thus, it converges to $E_*.$ Therefore any trajectory starting in $D_1$ or $D_2$ will approach to the equilibrium $E_*$ as $t\to\infty$.
\end{proof}

\begin{theorem}
Assume that $(\eta,\theta) \in \mathcal{R}$ and the quantity $A$ given in \eqref{criticality} is positive. Let $P(u_m,v_m)$, $Q(u_M,v_M)$ be the two fold points of the critical manifold $M_{20},$ and $\delta_{\mathrm{SNL}}<\delta<\delta_H$ then for $0 < \epsilon \ll 1$ sufficiently small, the system has two relaxation oscillations, where the outer cycle is a stable relaxation oscillation and the inner cycle is an unstable relaxation oscillation.
\end{theorem}
\begin{proof}
The proof of this theorem is a direct consequence of Theorem \eqref{sf_canard explosion}. The stability of the canard cycles can be computed numerically by using the constant $A$ from \eqref{criticality}.
\end{proof}

\section{Numerical Illustration}\label{sec6}
\textcolor{black}{
For $\epsilon$ sufficiently small i.e., $0 < \epsilon \ll 1$, the asymptotic expansion of the first Lyapunov coefficient $L_1$ in the blow-up coordinates \eqref{scaling} is computed in eqn.~\eqref{Eq:L1}. It has been shown that the sign of the leading order term $A$ in the expansion of $L_1$ determines the criticality of the singular Hopf bifurcation for $\epsilon\to 0$. The family of canard cycles which emerge due to singular Hopf bifurcation grow from $\IO(\epsilon)$ amplitude to large amplitude relaxation oscillation within an exponential small change of the bifurcation parameter. Now, the criticality of the singular Hopf bifurcation will be changed for $\epsilon>0$ sufficiently small, if $A$ changes its sign and hence, the singular Hopf bifurcation will be degenerate if $A=0$.}

\textcolor{black}{
In such a case, the system undergoes a codimension-2 Bautin bifurcation and the Bautin bifurcation threshold is given by $(\delta, \theta)=(\delta_H, \theta_B)$. Now, keeping all other parameters fixed, in a small neighborhood of the Bautin threshold point $(\delta, \theta)=(\delta_H, \theta_B)$ in the $\delta-\theta$ plane, the existence of a single  canard cycle or two canard cycles (unstable and stable) and the coalescent and disappearance of the cycles through a saddle-node bifurcation of cycles can be observed. We call this singular Bautin bifurcation in the sense that the associated Hopf bifurcation is singular; and for $\epsilon>0$ sufficiently small, the inner unstable canard cycle is the small canard cycle but, the outer stable canard cycle grows to a big relaxation oscillation.}

To numerically validate the results obtained We fix the parameters $\eta,\,\theta$ throughout our simulations such that \textbf{P4:} $\theta<\min\{1,\eta\}.$ We choose the hypothetical values from the above domain as $\theta=0.05$ and $\eta=0.176.$
 For this set of parameter values, we obtain two real roots of the quartic equation \eqref{dphi} which are the $u$ components of the fold points of the critical manifold. Thus, the fold points of $M_{20}$ are $$P(u_m,v_m) = (0.2375,0.2145),\,\text{and}\,\,Q(u_M,v_M) = (0.5359,0.235).$$ 
The parameter $\delta$ is chosen as the bifurcation parameter throughout the text to track the different singular bifurcations. Therefore, depending on the value of $\delta$ the coexistence equilibrium $E_*$ either lies on $S^l_0,$ $S^m_0$ or on $S^r_0$ part of critical manifold, $M_{20}.$ If $E_*$ coincides with either of the fold points, for some value of $\delta,$ then those thresholds determine a critical singular threshold. That is, for $\delta_*=0.2426879409,$ the predator nullcline intersects the critical manifold $M_{20}$ at $P(u_m,v_m) = (0.2375,0.2145).$ The Jacobian matrix evaluated at $Q$ has purely imaginary complex conjugate eigenvalues. Thus, from the condition of Hopf bifurcation, the system \eqref{poly_sys} undergoes a singular Hopf bifurcation, and the quantity $A$ given by the equation \eqref{criticality} is $$A \approx 2.796\times 10^{-7}.$$ This implies that the Hopf bifurcation is degenerate and subcritical. Therefore, the value of $\delta$ gives the threshold for singular Bautin bifurcation. As a result, the canard cycle originating from singular Hopf bifurcation is unstable, and the maximal canard curve is given by the equation \ref{sf_maximal canard}
\begin{equation}
 \begin{aligned}
       \delta_C & \approx \delta_*- \mathcal{O}(10^{-9}).
 \end{aligned}
\end{equation}
Since the transition from small amplitude unstable canard cycle to unstable maximal canard occurs in an extremely small parameter interval, it is difficult to catch this phenomenon through numerical simulation. However, we check for the second Lyapunov coefficient $L_2$ to confirm that this is a point of singular Bautin bifurcation. From the expression \eqref{second_Lyapunov_coefficient} and \eqref{second Lyapunov coefficient_2} we obtain the leading order coefficient of $L_2$ as $$B = -0.004.$$ 

\begin{figure}
    \centering
   \subfloat[]{\includegraphics[width=5.5cm,height=7cm]{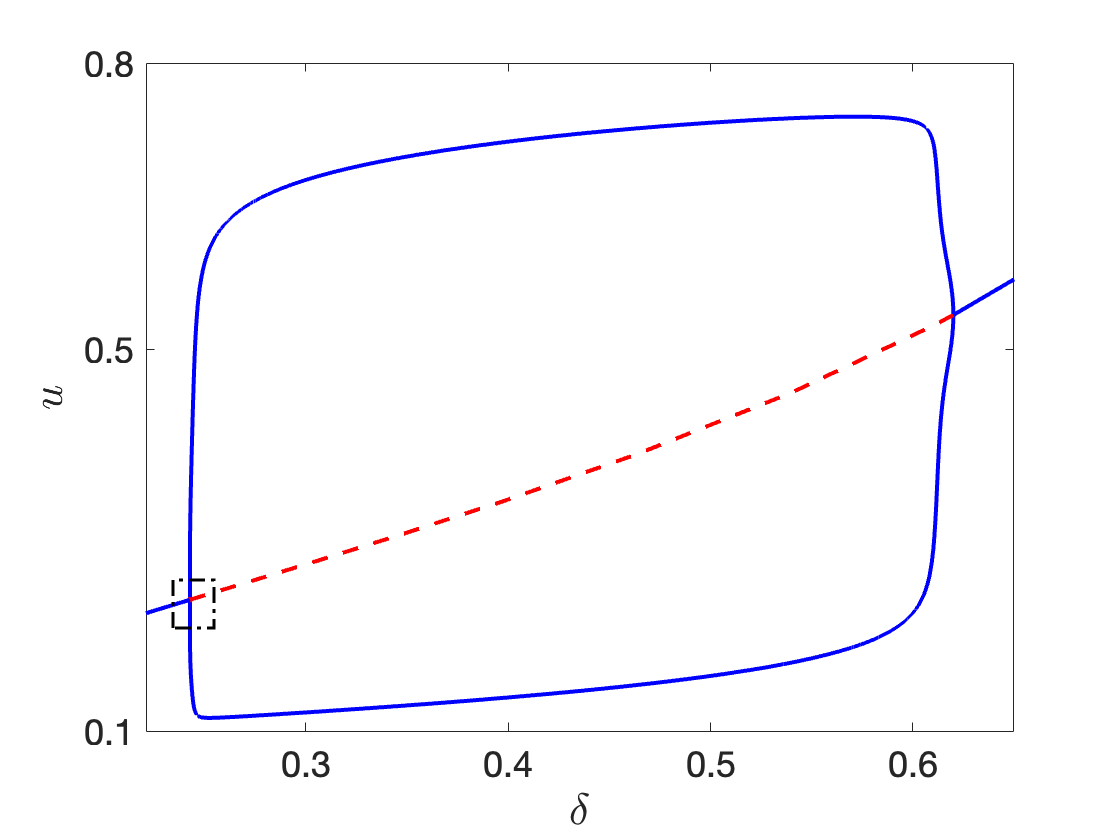}}
   \subfloat[]{\includegraphics[width=5.5cm,height=7cm]{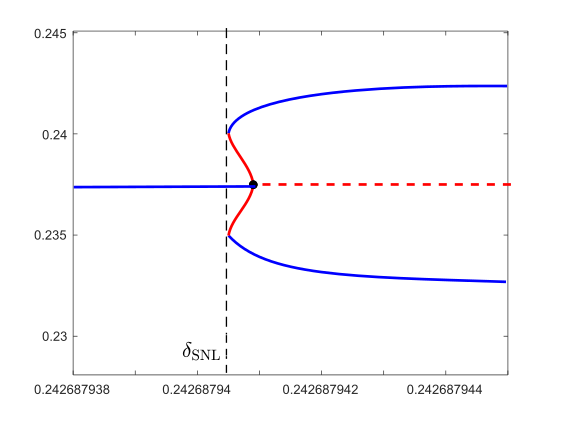}}
    \caption{(a) The bifurcation diagram of the system for $\eta=0.176,\,\theta=0.05,\,\epsilon=0.005.$ The zoomed-in diagram of the dotted rectangular box is shown in (b). The red broken and continuous line corresponds to the unstable equilibrium and unstable canard cycle, respectively. The blue continuous line represents the maxima and minima of the stable component (equilibrium and limit cycle). The black dot is the Singular Hopf point and the broken line is the threshold for saddle-node bifurcation of limit cycles $(\delta_{\mathrm{SNL}}).$}
    \label{fig:sub_critical_canardExplosion}
\end{figure}

The parameter $\delta$ is again varied such that the predator nullcline coincides with the prey nullcline at the point $Q(u_M,v_M)$ when $\delta=0.62.$ At this point the system undergoes a singular Hopf bifurcation which is supercritical, since $$A=-0.1055<0.$$
Thus, a stable canard cycle originates from this point and exists for $0.24268<\delta<0.62.$ Since the direction of the Hopf bifurcation is reversed in the maximum and minimum point of the critical manifold so is the direction of the canard explosion. Therefore, in this case, we observe the transition from small amplitude canard cycle to relaxation oscillation whenever $\delta<\delta_*$ and it continues in this direction. From theorem \eqref{sf_canard explosion}, we conclude that the family of canard cycles that emerges from this point are stable. Therefore, theorem \eqref{unique_RO} states  that there exists a unique relaxation oscillation. In Fig.~\eqref{fig:sub_critical_canardExplosion}, we plot the maximum and minimum components of the variable $u$ with respect to the parameter $\delta.$ The bifurcation diagram shows two Hopf bifurcation thresholds (cf.~Fig.\eqref{fig:sub_critical_canardExplosion}a). The zoomed-in diagram of the extreme left dotted box is shown in Fig.~\eqref{fig:sub_critical_canardExplosion}(b). It shows that in an extremely small neighborhood of the Bautin bifurcation, $\delta_{\mathrm{SNL}}<\delta<\delta_*,$ the small unstable canard cycles (red) and the stable canard cycles (blue) exists and they coalesce at the saddle-node bifurcation of the limit cycle. The threshold for $\delta$ is denoted as $\delta_{\mathrm{SNL}}$ and marked by a black broken line.

\section{Discussion}\label{sec7}
In this paper, we have considered a singularly perturbed planar predator-prey system with Holling type III functional response, where the prey population growth is affected by the weak Allee effect, and the prey reproduces at a significantly higher rate than the predator. The chosen model in the slow-fast framework admits explicit components for coexistence equilibrium, which destabilizes through supercritical and subcritical Hopf bifurcation for specific parameters. The Bautin bifurcation point separates the subcritical, and supercritical branches of the Hopf bifurcation curve. The stable and unstable limit cycles bifurcated through supercritical and subcritical Hopf bifurcation disappear through a saddle-node bifurcation of limit cycles.

\noindent We provide a thorough mathematical analysis of the system by using tools namely geometric singular perturbation theory, normal form theory of slow-fast systems, and blow-up technique to study a wide range of rich and complex nonlinear dynamics, such as singular Hopf bifurcations, singular Bautin bifurcations, canard cycles, canard explosion in the neighbourhood of subcritical and supercritical Hopf bifurcation. We have identified the occurrence of at most two relaxation oscillations via canard cycles known as the ``boom and bust cycle" \cite{rinaldi1992slow,sadhu2022analysis}. Important ecological ramifications stemming from the existence of relaxation oscillation and the onset of canard explosion are also an important part of the discussion. Using the entry-exit function, we have investigated that relaxation oscillations indeed occur in this slow-fast system, which suggests the possibility of the coexistence of both populations with a predictable pattern of rapid population growth and contractions over a significantly short time interval. On the other hand, the canard explosion is a surprising and fascinating event that happens in a tiny exponential region of the parameter $\delta$. From a biological perspective, this canard explosion can be seen as an indicator of an impending regime shift brought on by an exponentially modest change in parameter $\delta$.

\noindent The main contribution of this work is two–fold: 
 \begin{enumerate}
     \item We have derived the analytic form of the second Lyapunov coefficient and presented a thorough study of the singular Bautin bifurcation for a system under a slow-fast framework for the first time in the literature (see \cite{krupa2001extending,krupa2001relaxation,kuehn2015multiple}). The normal form for singular Bautin bifurcation has been derived by explicitly finding the locally invertible parameter-dependent transformations. To the best of our knowledge, no literature provides an explicit derivation of locally invertible parameter-dependent transformations. 
     \item We have derived the expansions of the first and second Lyapunov coefficients. It has been observed that the leading order coefficient of $L_1$, i.e., the coefficient of $\sqrt{\epsilon}$ involved with the expression of $L_1$, determines the criticality of the singular Hopf bifurcation, whereas the leading order coefficient of $L_2$, i.e., the coefficient of $\epsilon^{\frac{3}{2}}$ in $L_2$, determines the criticality of the singular Bautin bifurcation.
     \end{enumerate}
It is important to mention here that this approach would be applicable to investigate Bautin bifurcation for a wide range of slow-fast systems, including but not limited to biological, physical, ecological, and medical fields. Further, how this approach can be extended for higher dimensional slow-fast systems is a non-trivial open problem.

\bibliographystyle{siamplain}
	\bibliography{sl_fast}

\end{document}